\documentclass[cochineal]{roalart}

\author{Roberto Albesiano}
\thanks{This paper was partially supported by Simons Foundation International, LTD.\\ \indent This preprint has not undergone peer review or any post-submission improvements or corrections. The Version of Record of this article is published in \emph{Math.~Z.}, and is available online at \url{https://doi.org/10.1007/s00209-024-03435-6}}
\address{Mathematics Department, Stony Brook University}
\email*{ralbesiano@math.stonybrook.edu}
\urladdr{https://www.math.stonybrook.edu/~ralbesiano/}

\title{A degeneration approach to\\ Skoda's Division Theorem}

\begin{document}

\begin{abstract}
  We prove a Skoda-type division theorem via a degeneration argument. The proof is inspired by B.~Berndtsson and L.~Lempert's approach to the $L^2$ extension theorem and is based on positivity of direct image bundles. The same tools are then used to slightly simplify and extend the proof of the $L^2$ extension theorem given by Berndtsson and Lempert.
\end{abstract}

\maketitle

\section{Introduction}
This paper employs degeneration methods to prove an $L^2$ division theorem. The proof is inspired by a similar technique by B.~Berndtsson and L.~Lempert to give a new proof of a sharp $L^2$ extension theorem \cite{BerndtssonLempert2016}. In the process, we discover a tool that slightly simplifies and extends the proof of $L^2$ extension by Berndtsson and Lempert.

In the following $H^0(X,V)$ always denotes the space of holomorphic sections of the holomorphic vector bundle $V \rightarrow X$.

\begin{maintheorem}[\textbf{$L^2$ division}]\label{thm:SkodaBL}
  Let $X$ be a Stein manifold and let $E,G \rightarrow X$ be holomorphic line bundles with \paren{singular} Hermitian metrics $\e^{-\varphi}$ and $\e^{-\psi}$, respectively. Fix $h = (h_1, \dots, h_r) \in H^0(X, (E^* \otimes G)^{\oplus r})$ and $1 < \alpha < \frac{r+1}{r-1}$. Assume that
  \[
    \I \ddbar \varphi \geq \frac{\alpha(r-1)}{\alpha(r-1)+1} \I \ddbar \psi.
  \]
  Then for any holomorphic section $g \in H^0(X,G \otimes K_X)$ such that
  \[
    \norm{g}_G^2 := \int_X \frac{|g|^2 \e^{-\psi}}{(|h|^2 \e^{-\psi+\varphi})^{\alpha(r-1)+1}} < +\infty
  \]
  there is a holomorphic section $f = (f_1, \dots, f_r) \in H^0(X, E^{\oplus r} \otimes K_X)$ such that
  \[
    g = h \dototimes f := h_1 \otimes f_1 + \dots + h_r \otimes f_r
  \]
  and
  \[
    \begin{split}
      \norm{f}_{E^{\oplus r}}^2 &:= \int_X \frac{|f|^2 \e^{-\varphi}}{(|h|^2 \e^{-\psi+\varphi})^{\alpha(r-1)}} \\
      &\leq r \frac{\alpha}{\alpha-1} \int_X \frac{|g|^2 \e^{-\psi}}{(|h|^2 \e^{-\psi+\varphi})^{\alpha(r-1)+1}} = r \frac{\alpha}{\alpha-1} \norm{g}_G^2.
    \end{split}
  \]
\end{maintheorem}

If the number of generators $r$ is at most $\dim X + 1$, Theorem \ref{thm:SkodaBL} almost recovers the line bundle version of Skoda's Theorem (see Theorem \ref{thm:full-division}). We say ``almost'' because in Theorem \ref{thm:SkodaBL} we have $\norm{f}_{E^{\oplus r}}^2 \leq \frac{r\alpha}{\alpha-1} \norm{g}_G^2$ rather than $\norm{f}_{E^{\oplus r}}^2 \leq \frac{\alpha}{\alpha-1} \norm{g}_G^2$. Still, even though stronger results are known, our intent is to emphasize the technique used to prove Theorem \ref{thm:SkodaBL}. Both the original proof of Skoda \cite{Skoda1972} and more recent generalizations \cite{Varolin2008} are based on functional analysis and the Bochner formula, plus some very careful linear algebra estimates. Here instead we present a degeneration argument based on Berndtsson's Theorem on the positivity of direct image bundles \cite{Berndtsson2009}. Even though this degeneration approach is aesthetically pleasing, it is not clear if it can recover the theorem in full generality. Part of the goal of this project is to see if one can use degeneration techniques to achieve the same results obtained by the ostensibly more powerful $L^2$ methods.

The general philosophy is inspired by Berndtsson and Lempert's proof of the $L^2$ extension theorem \cite{BerndtssonLempert2016, Lempert2017} and by T.~Ohsawa's proof of a Skoda-type division theorem as a corollary of the Ohsawa--Takegoshi $L^2$ extension theorem. Ohsawa indeed remarks that the division problem can be reformulated as an extension problem on the projectivizations of the dual bundles (see \cite{Ohsawa2002}, \cite{Ohsawa2004} and \cite[Section 3.2]{Ohsawa2015}). It is thus natural to wonder whether a Skoda-type theorem could be proved directly using techniques akin to \cite{BerndtssonLempert2016}.

The main idea is to look at all possible linear combinations $v_1 \otimes f_1(x) + \dots + v_r \otimes f_r(x)$. Then one constructs a positively curved family of metrics that at one extreme ``localizes'' the problem at the point of interest $v = h(x)$ and at the other extreme retrieves the usual $L^2$-norm for $f$. Near $h(x)$ the optimal solution to the division problem is somehow ``trivial''; for instance, if $h(x) = (h_1(x), 0, \dots, 0)$, one takes $f(x) = (g(x) h_1(x)^{-1}, 0, \dots, 0)$. The positivity of the direct image bundle \cite{Berndtsson2009} will then imply that one can control $\norm{f}_{E^{\oplus r}}$ by the norm of the trivial solution near $h(x)$. This last step is an instance of a more general extrapolation technique for estimating operator norms under suitable curvature conditions \cite{Lempert2017}.

We start in Section \ref{sec:background} by recalling the definition of singular Hermitian metric (Definition \ref{def:shm}) and the fundamental theorem of Berndtsson on direct images (Theorem \ref{thm:Berndtsson}). We also recall the full statements of $L^2$ division (Theorem \ref{thm:full-division}) and $L^2$ extension (Theorem \ref{thm:full-extension}) as obtained with the standard $L^2$ technique. Next, in Section \ref{sec:calculus-lemma} we prove a calculus lemma that will be used in the proof of both Theorem \ref{thm:SkodaBL} and Theorem \ref{thm:L2-extension} below. The majority of this paper is then devoted to the proof of Theorem \ref{thm:SkodaBL}. In Section \ref{sec:reductions} we reduce to the case in which $h_1, \dots, h_r$ have no common zeros and $X$ is a relatively compact domain in some Stein manifold. Section \ref{sec:dual-formulation} reformulates the division problem as an estimate on the dual norm of some special functionals on $H^0(X, E^{\oplus r} \otimes K_X)$. The main argument for Theorem \ref{thm:SkodaBL} is in Section \ref{sec:main-argument}: after setting up the degeneration in \ref{subsec:setup}, we compute the extrema of the family of metrics in \ref{subsec:extrema} and we apply Berndtsson's Theorem in \ref{subsec:convexity} to conclude the proof.

It turns out that the same schema can be used to prove the following version of $L^2$ extension.

\begin{maintheorem}[\textbf{$L^2$ extension}]\label{thm:L2-extension}
  Let $X$ be a Stein manifold and $Z \subset X$ an analytic hypersurface. Let $L_Z \rightarrow X$ be the holomorphic line bundle associated to $Z$, with $T \in H^0(X, L_Z)$ such that $Z = (T=0)$ and $\dif T|_Z$ generically non-zero. Assume moreover that $L_Z$ carries a \paren{singular} Hermitian metric $\e^{-\lambda}$ such that $\e^{-\lambda}|_Z \not\equiv +\infty$ and $\sup_X |T|^2 \e^{-\lambda} \leq 1$.
  
  Let $L \rightarrow X$ be a line bundle with \paren{singular} Hermitian metric $\e^{-\varphi}$ such that 
  \[
    \I \ddbar \varphi \geq - \I \ddbar \lambda \quad \text{and} \quad \I \ddbar \varphi \geq \delta \I \ddbar \lambda
  \]
  for some $\delta > 0$. Then for any holomorphic section $f \in H^0(Z, L|_Z \otimes K_Z)$ such that 
  \[
    \norm{f}_Z^2 := \int_Z |f|^2 \e^{-\varphi} < +\infty
  \]
  there is a holomorphic section $F \in H^0(X, L \otimes L_Z \otimes K_X)$ such that $F|_Z = f \wedge \dif T$ and 
  \[
    \norm{F}_X^2 := \int_X |F|^2 \e^{-\lambda - \varphi} \leq \pi \left( 1 + \frac{1}{\delta} \right) \int_Z |f|^2 \e^{-\varphi} = \pi \left( 1 + \frac{1}{\delta} \right) \norm{f}_Z^2.
  \]
\end{maintheorem}

We want to emphasize that the sole difference between this last statement and the actual $L^2$ extension theorem (Theorem \ref{thm:full-extension}) is that here we require
\begin{equation}\label{eq:l2-curvature-condition}
  \begin{cases}
    \I \ddbar \varphi \geq -\I \ddbar \lambda \\
    \I \ddbar \varphi \geq \delta \I \ddbar \lambda
  \end{cases}
\end{equation}
instead of the weaker 
\[
  \begin{cases}
    \I \ddbar \varphi \geq 0 \\
    \I \ddbar \varphi \geq \delta \I \ddbar \lambda
  \end{cases}.
\]
In particular, the two conditions are clearly the same in the special case $\I\ddbar\lambda \geq 0$ and Theorem \ref{thm:L2-extension} recovers the hyperplane case of Berndtsson and Lempert's Theorem \cite[Theorem 3.8]{BerndtssonLempert2016}.

A key role in the proof of Theorem \ref{thm:L2-extension}, which is summarized in Section \ref{sec:L2-extension}, will be played by the seemingly innocuous Lemma \ref{lemma:calculusEstimate}. This slightly improved version of \cite[Lemma 3.4]{BerndtssonLempert2016} is indeed what allows us to prove the theorem under the assumption \eqref{eq:l2-curvature-condition} rather than $\I\ddbar\lambda\geq0$, and at the same time simplify the proof.

The assumption that $\codim_X Z = 1$ in Theorem \ref{thm:L2-extension} makes the statement and the proof cleaner but is not really needed. As remarked in Section \ref{sec:L2-higherCodim} (in particular Theorem \ref{thm:L2-nonAdjoint}), in fact minimal modifications of the same arguments prove analogous statements for $Z$ of higher codimension.

\vspace{0.5\baselineskip}
\paragraph*{\textbf{Acknowledgements}}
I am grateful to Dror Varolin for bringing this topic to my attention, for many helpful discussions and a lot of encouragement. I also thank Bo Berndtsson, L\'{a}szl\'{o} Lempert, Christian Schnell, and Xu Wang for providing useful comments and suggestions. Finally, I thank the anonymous referees for their helpful observations.

\section{Background and notation}\label{sec:background}
Let $L$ be a holomorphic line bundle over a complex manifold $X$.

\begin{definition}[\cite{Demailly1992}]\label{def:shm}
  A \emph{singular Hermitian metric} for $L$ is a measurable section $\e^{-\phi}$ of $L^* \otimes (L^*)^\dagger \rightarrow X$ such that for any holomorphic frame $\xi$ of $L$ over $U \subset X$ the measurable function $|\xi|^2\e^{-\phi}$ is non-negative and $\log(|\xi|^2\e^{-\phi}) \in L^1_{\text{loc}}(U)$.
\end{definition}

Note that, according to this definition, smooth Hermitian metrics are in fact singular Hermitian metrics.

Assume now that $L$ is a holomorphic line bundle over an ambient Stein manifold $Y$ and let $X$ be relatively compact in $Y$. Consider the trivial fibration $X \times \mathbb{D} \rightarrow \mathbb{D}$, where $\mathbb{D}$ is the complex unit disk. Let $p: X \times \mathbb{D} \rightarrow X$ be the projection on the first factor and $\e^{-\phi}$ be a Hermitian metric for $p^*L \rightarrow X \times \mathbb{D}$ that is smooth up to the vertical boundary of $X \times \mathbb{D}$. For each $\tau \in \mathbb{D}$, define the Hilbert space 
\[
  \mathcal{H}_\tau(\e^{-\phi}) := \left\{ f \in H^0(X, L \otimes K_X) \, \middle| \, \, \norm{f}_\tau^2 := \int_X |f|^2 \e^{-\phi_\tau} < +\infty \right\},
\]
where $\e^{-\phi_\tau}$ is the restriction of $e^{-\phi}$ to $p^* L|_{X \times \{\tau\}}$.

Since the metric $\e^{-\phi}$ is smooth up to the boundary of $X$, all these Hilbert spaces are in fact trivially the same as vector spaces but have norms that vary with $\tau$. Thus, they form a trivial vector bundle with a non-trivial metric. The following theorem is due to B.~Berndtsson.

\begin{theorem}[\cite{Berndtsson2009}]\label{thm:Berndtsson}
  Assume that $\e^{-\phi}$ is a smooth Hermitian metric for $p^*L \rightarrow X \times \mathbb{D}$ with non-negative curvature and let $\xi \not\equiv 0$ be a holomorphic section of the vector bundle $\mathcal{H}_\tau(\e^{-\phi})^* \rightarrow \mathbb{D}$. Then the function 
  \[
    \mathbb{D} \ni \tau \longmapsto \log\norm{\xi_t}_{\tau,*}^2
  \]
  is subharmonic.
\end{theorem}

Thus in particular $\tau \mapsto \log\norm{\xi}_{t,*}^2$ is subharmonic for any fixed non-zero $\xi \in H^0(X, L \otimes K_X)^*$ with finite $L^2$-norm.

Next, we recall the complete statements of $L^2$ division and $L^2$ extension obtained from the original proofs based on $L^2$ methods.

\begin{theorem}[\textbf{$L^2$ division}]\label{thm:full-division}
  Let $X$ be a Stein manifold of complex dimension $n$ and let $E,G \rightarrow X$ be holomorphic line bundles with \paren{singular} Hermitian metrics $\e^{-\varphi}$ and $\e^{-\psi}$, respectively. Fix $h = (h_1, \dots, h_r) \in H^0(X, (E^* \otimes G)^{\oplus r})$ and $\alpha>1$. Let $q := \min(r-1,n)$ and assume that
  \[
    \I \ddbar \varphi \geq \frac{\alpha q}{\alpha q+1} \I \ddbar \psi.
  \]
  Then, for any holomorphic section $g \in H^0(X,G \otimes K_X)$ such that
  \[
    \norm{g}_G^2 := \int_X \frac{|g|^2 \e^{-\psi}}{(|h|^2 \e^{-\psi+\varphi})^{\alpha q + 1}} < +\infty,
  \]
  there is a holomorphic section $f = (f_1, \dots, f_r) \in H^0(X, E^{\oplus r} \otimes K_X)$ such that
  \[
    g = h \dototimes f := h_1 \otimes f_1 + \dots + h_r \otimes f_r
  \]
  and
  \[
    \begin{split}
      \norm{f}_{E^{\oplus r}}^2 &:= \int_X \frac{|f|^2 \e^{-\varphi}}{(|h|^2 \e^{-\psi+\varphi})^{\alpha q}} \\
      &\leq \frac{\alpha}{\alpha-1} \int_X \frac{|g|^2 \e^{-\psi}}{(|h|^2 \e^{-\psi+\varphi})^{\alpha q + 1}} = \frac{\alpha}{\alpha-1} \norm{g}_G^2
    \end{split}
  \]
  \paren{see \cite[Th\'eor\`eme 6.2]{Demailly1982} and \cite[Theorem 2.1]{Varolin2008}}.
\end{theorem}

Once more, in Theorem \ref{thm:SkodaBL} the constant $q$ is everywhere replaced by $r-1$, the constant $\frac{\alpha}{\alpha-1}$ is replaced by $\frac{r\alpha}{\alpha-1}$, and $\alpha$ cannot be greater than $\frac{r+1}{r-1}$. Consequently, Theorem \ref{thm:SkodaBL} is strictly weaker than Theorem \ref{thm:full-division}.

\begin{theorem}[\textbf{$L^2$ extension}]\label{thm:full-extension}
  Let $X$ be a Stein manifold and $Z \subset X$ an analytic hypersurface. Let $L_Z \rightarrow X$ be the holomorphic line bundle associated to $Z$, with $T \in H^0(X, L_Z)$ such that $Z = (T=0)$ and $\dif T|_Z$ generically non-zero. Assume moreover that $L_Z$ carries a \paren{singular} Hermitian metric $\e^{-\lambda}$ such that $\e^{-\lambda}|_Z \not\equiv +\infty$ and $\sup_X |T|^2 \e^{-\lambda} \leq 1$.
  
  Let $L \rightarrow X$ be a line bundle with \paren{singular} Hermitian metric $\e^{-\varphi}$ such that 
  \[
    \I \ddbar \varphi \geq 0 \quad \text{and} \quad \I \ddbar \varphi \geq \delta \I \ddbar \lambda
  \]
  for some $\delta > 0$. Then for any holomorphic section $f \in H^0(Z, L|_Z \otimes K_Z)$ such that 
  \[
    \norm{f}_Z^2 := \int_Z |f|^2 \e^{-\varphi} < +\infty
  \]
  there is a holomorphic section $F \in H^0(X, L \otimes L_Z \otimes K_X)$ such that $F|_Z = f \wedge \dif T$ and 
  \[
    \norm{F}_X^2 := \int_X |F|^2 \e^{-\lambda - \varphi} \leq \pi \left( 1 + \frac{1}{\delta} \right) \int_Z |f|^2 \e^{-\varphi} = \pi \left( 1 + \frac{1}{\delta} \right) \norm{f}_Z^2
  \]
  \paren{see \cite[Theorem 1]{Blocki2013} and \cite[Theorem 2.1]{GuanZhou2015}}.
\end{theorem}

Again, the difference between Theorem \ref{thm:L2-extension} and Theorem \ref{thm:full-extension} is that in the latter we have a slightly weaker requirement on the curvature of $\e^{-\varphi}$.

Here we have only considered the case of extension from a hypersurface since in this case all objects are always naturally defined. See \cite{BerndtssonLempert2016,Ohsawa2001,Manivel1993,Demailly2000} for the corresponding results in higher codimension.

\section{A calculus lemma}\label{sec:calculus-lemma}
We now establish the following lemma, which plays a key role in the proofs of both Theorem \ref{thm:SkodaBL} and Theorem \ref{thm:L2-extension}. The result is a slight modification of Lemma 3.4 in \cite{BerndtssonLempert2016}.

\begin{lemma}\label{lemma:calculusEstimate}
  Let $\nu: (-\infty,0] \rightarrow \mathbb{R}_+$ be an increasing function such that
  \[
    \lim_{t \rightarrow -\infty} \e^{-Bt} \nu(t) = A < +\infty
  \]
  for some $B > 0$. Then, for all $p>B$,
  \[
    \lim_{t \rightarrow -\infty} \e^{-Bt} \int_t^0 \e^{-p(s-t)} \dif \nu(s) = \frac{AB}{p-B}.
  \]
\end{lemma}

\begin{remark}
  In contrast to Lemma 3.4 in \cite{BerndtssonLempert2016}, we do not require $\nu$ to be bounded above by $A \e^{Bt}$ for \emph{all} $t < 0$. This weakened hypothesis allows us to obtain the more precise estimates needed for Theorem \ref{thm:SkodaBL} and Theorem \ref{thm:L2-extension}.
\end{remark}

\begin{proof}
  Integrating by parts one gets 
  \[
    \begin{split}
      \e^{-Bt}&\int_t^0 \e^{-p(s-t)} \dif\nu(s) = \e^{(p-B)t} \int_t^0 \e^{-ps} \dif\nu(s) \\
      &= \e^{(p-B)t} \left[ \nu(0) - \e^{-pt} \nu(t) + p \int_t^0 \e^{-ps} \nu(s) \dif s \right].
    \end{split}
  \]

  By the assumptions we have 
  \[
    \lim_{t \to -\infty} \e^{(p-B)t} \left( \nu(0) - \e^{-pt}\nu(t) \right) = -A.
  \]
  Moreover, for any $\varepsilon>0$ there is $t_\varepsilon < 0$ such that $(A-\varepsilon)\e^{Bt} \leq \nu(t) \leq (A+\varepsilon) \e^{Bt}$ for all $t \leq t_\varepsilon$. Then 
  \[
    \begin{split}
      &\int_t^0 \e^{-ps}\nu(s) \dif s \leq \int_{t_\varepsilon}^0 \e^{-ps}\nu(s) \dif s + (A+\varepsilon) \int_t^{t_\varepsilon} \e^{-(p-B)s} \dif s \\
      &\quad \leq C_\varepsilon + \frac{A+\varepsilon}{p-B} \left( \e^{-(p-B)t} - \e^{-(p-B)t_\varepsilon} \right) = C'_\varepsilon + \frac{A+\varepsilon}{p-B} \e^{-(p-B)t}
    \end{split}
  \]
  and similarly
  \[
    \int_t^0 \e^{-ps} \nu(s) \dif s \geq C''_\varepsilon + \frac{A-\varepsilon}{p-B} \e^{-(p-B)t},
  \]
  so that
  \[
    \frac{A-\varepsilon}{p-B} \leq \lim_{t \to -\infty} \e^{(p-B)t} \int_t^0 \e^{-ps}\nu(s) \dif s \leq \frac{A+\varepsilon}{p-B}.
  \]
  Since this holds for all $\varepsilon>0$ we conclude that 
  \[
    \lim_{t \to -\infty} \e^{(p-B)t} p \int_t^0 \e^{-ps}\nu(s) \dif s = \frac{Ap}{p-B}
  \]
  and then
  \[
    \lim_{t \to -\infty} \e^{-Bt} \int_t^0 \e^{-p(s-t)} \dif\nu(s) = -A + \frac{Ap}{p-B} = \frac{AB}{p-B},
  \]
  as wanted.
\end{proof}
We can now move to the proof of Theorem \ref{thm:SkodaBL}.

\section{Preliminary reductions for Theorem \ref{thm:SkodaBL}}\label{sec:reductions}
\paragraph{\textbf{No base locus}}
We can assume that the sections $h_1, \dots, h_r$ have no common zeros. Indeed, let $D$ be the zero-set of $h_r$. Then $X \setminus D$ is again Stein and $h|_{X \setminus D}$ has no zeros. Assuming that Theorem \ref{thm:SkodaBL} holds for $\{h_1 = \dots = h_r = 0\} = \varnothing$, we obtain $\tilde{f} \in H^0(X \setminus D, (E^{\oplus r} \otimes K_X)|_{X \setminus D})$ such that
\[
  g|_{X \setminus D} = h|_{X \setminus D} \dototimes \tilde{f}
\]
and
\[
  \int_{X \setminus D} \frac{|\tilde{f}|^2 \e^{-\varphi}}{(|h|^2 \e^{-\psi+\varphi})^{\alpha (r-1)}} \leq r \frac{\alpha}{\alpha-1} \norm{g}^2_G < +\infty.
\]
As $h$ is bounded on any bounded chart $U \subset\subset X$,
\[
  \int_{U \setminus D} |\tilde{f}|^2 \e^{-\varphi} \leq C \int_{U \setminus D} \frac{|\tilde{f}|^2 \e^{-\varphi}}{(|h|^2 \e^{-\psi+\varphi})^{\alpha(r-1)}} \leq C r \frac{\alpha}{\alpha-1} \norm{g}^2_G < +\infty,
\]
where $C > 0$ depends on $U$, $h$ and $\alpha(r-1)$. Hence, by Riemann's Removable Singularities Theorem, $\tilde{f}$ extends to $f \in H^0(X, E^{\oplus r} \otimes K_X)$. As $D$ has measure 0,
\[
  \int_{X} \frac{|f|^2 \e^{-\varphi}}{(|h|^2 \e^{-\psi+\varphi})^{\alpha(r-1)}} = \int_{X \setminus D} \frac{|\tilde{f}|^2 \e^{-\varphi}}{(|h|^2 \e^{-\psi+\varphi})^{\alpha(r-1)}} \leq r \frac{\alpha}{\alpha-1} \norm{g}^2_G < +\infty,
\]
and, because $h \dototimes f$ and $g$ coincide on the open set $X \setminus D$, we have $h \dototimes f = g$ everywhere on $X$, solving the division problem.

\begin{remark}
  The same argument proves Theorem \ref{thm:SkodaBL} when $X$ is essentially Stein, given that it has been proved for Stein manifolds. Recall that a manifold $X$ is \emph{essentially Stein} if there is a divisor $D$ such that $X \setminus D$ is Stein. For instance, projective manifolds are essentially Stein.
\end{remark}

\paragraph{\textbf{$X$ bounded pseudoconvex}}
We can reduce $X$ to a relatively compact domain in some larger Stein manifold, and say that sections extend up to the boundary of $X$ (or are defined on $\bar{X}$) if they extend to a neighborhood of $X$ in the ambient Stein manifold. We can also assume that $\omega$ and $E,G$ extend to the ambient manifold ($E,G$ holomorphically) and that the metrics $\e^{-\varphi}$ and $\e^{-\psi}$ are smooth. If the result is proved under these assumptions, then the universality of the bounds yields the general case by standard weak-$*$ compactness theorems, Lebesgue-type limit theorems and approximation results for singular Hermitian metrics on Stein manifolds (see the first paragraph in Section 3 of \cite{BerndtssonLempert2016}).

\section{Dual formulation of the division problem}\label{sec:dual-formulation}
Fix a section $g \in H^0(X,G \otimes K_X)$ to be divided. We may assume, after possibly shrinking $X$, that $g$ is holomorphic up to the boundary of $X$. Let $\gamma: E^{\oplus r} \otimes K_X \rightarrow G \otimes K_X$ be defined by
\[
  \gamma(e_1,\dots,e_r) := h_1 \otimes e_1 + \dots + h_r \otimes e_r.
\]

\begin{proposition}\label{prop:solExistence}
  There exists $f = (f_1, \dots, f_r) \in H^0(X, E^{\oplus r} \otimes K_X)$ such that
  \[
    g = h_1 \otimes f_1 + \dots + h_r \otimes f_r = h \dototimes f
  \]
  and
  \[
    \norm{f}^2_{E^{\oplus r}} < +\infty.
  \]
\end{proposition}

\begin{proof}
  Since $X$ is a relatively compact domain in a Stein manifold, any solution $f$ in the ambient manifold will restrict to a solution on $X$ with bounded $L^2$-norm. Hence, it suffices to show that for a Stein manifold $X$ there is a not-necessarily-$L^2$ solution of the division problem.

  As the $h_1, \dots, h_r$ have no common zeros, the map $\gamma$ is a surjective morphism of vector bundles and thus we have the short exact sequence of vector bundles
  \[
    0 \longrightarrow \ker \gamma \longrightarrow E^{\oplus r} \otimes K_X \longrightarrow G \otimes K_X \longrightarrow 0.
  \]
  The induced sequence in cohomology then yields
  \[
    0 \rightarrow H^0(X, \ker \gamma) \rightarrow H^0(X, E^{\oplus r} \otimes K_X) \rightarrow H^0(X,G \otimes K_X) \rightarrow H^1(X,\ker \gamma) = 0,
  \]
  where the last term on the right vanishes by Cartan's Theorem B. Hence, the map induced by $\gamma$ in cohomology is surjective, meaning that for any $g \in H^0(X,G \otimes K_X)$ we can find
  \[
    f = (f_1, \dots, f_r) \in H^0(X, E^{\oplus r} \otimes K_X)
  \]
  such that
  \[
    g = \gamma \circ f = h_1 \otimes f_1 + \dots + h_r \otimes f_r,
  \]
  proving the statement.
\end{proof}

Since there is a solution $f$ with finite $L^2$-norm, there is a (unique) solution $\tilde{f}$ with minimal $L^2$-norm. To prove Theorem \ref{thm:SkodaBL}, we thus need to estimate $\norm{\tilde{f}}_{E^{\oplus r}}$.

\begin{lemma}\label{lemma:minimalNorm}
  Let $f \in H^0(X,E^{\oplus r} \otimes K_X)$ be any solution to the division problem with finite $L^2$-norm. Then the solution $\tilde{f}$ with minimal $L^2$-norm has norm
  \[
    \norm{\tilde{f}}_{E^{\oplus r}}^2 = \sup_{\xi \in \ann  H^0(X,\ker\gamma)} \frac{|\xi(f)|^2}{\norm{\xi}_*^2},
  \]
  where $\norm{\cdot}_*$ is the norm for the dual Hilbert space $H^0(X, E^{\oplus r} \otimes K_X)^*$ and $\ann H^0(X,\ker\gamma)$ is the annihilator of $H^0(X,\ker\gamma)$, i.e.~all linear functionals on $H^0(X, E^{\oplus r} \otimes K_X)$ that vanish on $H^0(X,\ker\gamma)$.

  Moreover, one can restrict the supremum to functionals $\xi_\eta \in H^0(X, E^{\oplus r} \otimes K_X)^*$ of the form
  \[
    \xi_\eta(f) := (\gamma \circ f, \eta)_G = \int_X \frac{(h \dototimes f) \bar{\eta} \e^{-\psi}}{(|h|^2 \e^{-\psi+\varphi})^{\alpha(r-1)+1}},
  \]
  for $\eta \in C^\infty_c(X,G \otimes K_X)$ \paren{smooth compactly supported sections of $G \otimes K_X$}.
\end{lemma}

\begin{proof}
  Note first that the supremum is independent of the choice of the arbitrary $L^2$ solution $f$. Indeed, if $\xi \in \ann H^0(X,\ker\gamma)$ and $\gamma \circ f = \gamma \circ f' = g$, then by linearity $f - f' \in H^0(X,\ker\gamma)$, so that $\xi(f) = \xi(f')$.

  Next, we claim that $\tilde{f} \perp H^0(X,\ker\gamma)$. Indeed, if $k \in H^0(X,\ker\gamma)$, then $\gamma \circ (\tilde{f} + \varepsilon k) = g$ for all $\varepsilon \in \mathbb{C}$. As $\tilde{f}$ is the minimal norm solution, we have that
  \[
    \mathbb{C} \ni \varepsilon \longmapsto \norm{\tilde{f} + \varepsilon k}_{E^{\oplus r}}^2 = \norm{\tilde{f}}_{E^{\oplus r}}^2 + 2 \Re [(\tilde{f},k)_{E^{\oplus r}} \varepsilon] + O(|\varepsilon|^2)
  \]
  has minimum at $\varepsilon = 0$ (here $(\cdot,\cdot)_{E^{\oplus r}}$ denotes the $L^2$ inner product on $E^{\oplus r} \otimes K_X$). Hence $(\tilde{f}, k)_{E^{\oplus r}} = 0$.

  Finally, notice that if $k \in H^0(X,\ker \gamma)$ then
  \[
    \xi_\eta(k) = (\gamma \circ k, \eta)_G = 0,
  \]
  i.e.~$\xi_\eta \in \ann H^0(X,\ker \gamma)$. Conversely, if
  \[
    0 = \xi_\eta(f) = (\gamma \circ f, \eta)_G
  \]
  for all $\eta \in C^\infty_c(X,G \otimes K_X)$, then $\gamma \circ f = 0$. Hence
  \[
    \left\{ \xi_\eta \mid \eta \in C^\infty_c(X,G \otimes K_X) \right\} \subseteq \ann H^0(X, \ker\gamma)
  \]
  is dense and we may restrict to elements $\xi_\eta$ when computing the supremum.
\end{proof}

By Lemma \ref{lemma:minimalNorm}
\begin{equation}\label{eq:dualFormulation}
  \begin{split}
    \norm{\tilde{f}}_{E^{\oplus r}}^2 &= \sup_{\eta \in C^\infty_c(X,G \otimes K_X)} \frac{|(\gamma \circ f,\eta)|^2}{\norm{\xi_\eta}_*^2} \\
    &= \sup_{\eta \in C^\infty_c(X,G \otimes K_X)} \frac{|(g,\mathcal{P}\eta)|^2}{\norm{\xi_\eta}_*^2} \leq \norm{g}_G^2 \sup_{\eta \in C^\infty_c(X,G \otimes K_X)} \frac{\norm{\mathcal{P} \eta}_G^2}{\norm{\xi_\eta}_*^2},
  \end{split}
\end{equation}
where
\[
  \mathcal{P}: L^2(X,G \otimes K_X) \longrightarrow H^0(X,G \otimes K_X) \cap L^2(X, G \otimes K_X)
\]
denotes the Bergman projection. Therefore, to prove Theorem \ref{thm:SkodaBL} it suffices to prove that
\[
  \norm{\mathcal{P}\eta}_G^2 \leq r \frac{\alpha}{\alpha-1} \norm{\xi_\eta}_*^2
\]
for all $\eta \in C^\infty_c(X,G \otimes K_X)$.

\section{Proof of Theorem \ref{thm:SkodaBL}}\label{sec:main-argument}
\subsection{Setup}\label{subsec:setup}
Instead of working directly on the vector bundle $E^{\oplus r} \otimes K_X \rightarrow X$, we lift everything to the line bundle
\[
  L := \pr_X^* (E \otimes K_X) \otimes \pr_{\mathbb{P}_{r-1}}^* \mathcal{O}_{\mathbb{P}_{r-1}}(1) \longrightarrow X \times \mathbb{P}_{r-1},
\]
where $\mathcal{O}_{\mathbb{P}_{r-1}}(1)$ is the hyperplane bundle of $\mathbb{P}_{r-1}$ and $\pr_X$, $\pr_{\mathbb{P}_{r-1}}$ are the projections of $X \times \mathbb{P}_{r-1}$ on $X$, $\mathbb{P}_{r-1}$ respectively. Explicitly, fix once for all coordinates $v_1, \dots, v_r$ for $\mathbb{C}^r$ (descending to the homogeneous coordinates $[v_1:\dots:v_r]$ for $\mathbb{P}_{r-1}$) and declare the lift of a section $s \in H^0(X, E^{\oplus r} \otimes K_X)$ to be the section $\hat{s} \in H^0(X \times \mathbb{P}_{r-1}, L)$ defined by
\begin{equation}\label{eq:lift}
  \hat{s}(x,[v]) := v^* \cdot s(x) = v_1^* s_1(x) + \dots + v_r^* s_r(x) \in H^0(X \times \mathbb{P}_{r-1}, L),
\end{equation}
where $v_1^*, \dots, v_r^*$ are the dual coordinates of $v_1, \dots, v_r$ of $\mathbb{C}^r$.

Notice that the lift is a bijective map, since all sections of $L$ are of the form \eqref{eq:lift}. We can then lift the functionals $\xi_\eta \in H^0(X, E^{\oplus r} \otimes K_X)^*$ of Lemma \ref{lemma:minimalNorm} to functionals $\hat{\xi}_\eta \in H^0(X \times \mathbb{P}_{r-1}, L)^*$ defined as $\hat{\xi}_\eta(\hat{s}) := \xi_\eta(s)$ for all $\hat{s} \in H^0(X \times \mathbb{P}_{r-1}, L)$.

\begin{remark}
  One can interpret the lifted section $\hat{s}$ by thinking of the projective space $\mathbb{P}_{r-1}$ as parametrizing all possible choices of linear combinations (up to scaling). Hence, the value of the section $\hat{s}$ at $(x,[v])$ can be thought of (tautologically) as the linear combination parametrized by $v$ of the entries of the vector $s(x)$. What follows is essentially a procedure to ``single out'' the linear combination given by $[h(x)]$ (the equivalence class of $h(x)$ in $\mathbb{P}_{r-1}$).
\end{remark}

Next, we define a family of metrics for $L \rightarrow X \times \mathbb{P}_{r-1}$, parametrized by
\[
  \tau \in \mathbb{L} := \{ z \in \mathbb{C} \mid \Re z < 0\}.
\]
Toward this end, let
\[
  \chi_\tau(x,v) := \max \left( \log(|v|^2 |h(x)|^2 - |v \cdot h(x)|^2) \e^{-\psi+\varphi} - \Re \tau, \, \log |v|^2 |h(x)|^2 \e^{-\psi+\varphi} \right).
\]
Then, for $\sigma \in L_{(x,[v])}$, set
\[
  \mathfrak{h}_\tau(\sigma, \bar{\sigma})_{(x,[v])} := \frac{r!}{\pi^{r-1}} \e^{-(r-1) \Re \tau} \frac{|\sigma|^2 \e^{-\varphi}}{|v|^2} \left( |v|^2 \e^{-\chi_\tau} \right)^{\alpha(r-1)}.
\]

Notice that whether the maximum defining $\chi_\tau$ is attained by the first or the second entry is independent of the choice of the representative $v$ of $[v] \in \mathbb{P}_{r-1}$, and that the weight $|v|^2 \e^{-\chi_\tau}$ is a well-defined function on $X \times \mathbb{P}_{r-1}$. Notice also that $\chi_\tau$ depends only on $t := \Re \tau$ (as does $\mathfrak{h}_\tau$), so in the following we will write $\chi_t$ (and $\mathfrak{h}_t$) instead.

\begin{remark}
  As we shall soon see, the choice of $\chi_t$ is motivated as follows. For $t=0$, the maximum is realized by the second entry, so that
  \[
    (|v|^2 \e^{-\chi_t})^{\alpha(r-1)} = \frac{1}{(|h(x)|^2 \e^{-\psi+\varphi})^{\alpha(r-1)}},
  \]
  providing the weighting by the norm of $h$ in $\norm{\cdot}_{E^{\oplus r}}$.

  On the other hand, as $t \rightarrow -\infty$, the function $|v|^2 \e^{-\chi_t}$ gets very small at all $v \in \mathbb{P}_{r-1}$ that are not ``sufficiently aligned'' with $h(x)$. The precise sense of this statement will be more evident in Subsection \ref{subsec:extrema}, but for the moment note that $1 - \frac{|v \cdot h(x)|^2}{|v|^2 |h(x)|^2}$ constitutes a measurement of the angle between $v$ and $h(x)$.
\end{remark}

The family of metrics $\mathfrak{h}_\tau$ induces the family of $L^2$-norms
\[
  \norm{\sigma}_\tau^2 := \frac{r!}{\pi^{r-1}} \e^{-(r-1) \Re \tau} \int_{X \times \mathbb{P}_{r-1}} \frac{|\sigma|^2 \e^{-\varphi}}{|v|^2} \left( |v|^2 \e^{-\chi_\tau} \right)^{\alpha(r-1)} \wedge \dVFS,
\]
where $\dVFS$ is the (fixed) Fubini--Study volume form of $\mathbb{P}_{r-1}$.

We interpret the family $\mathfrak{h}_\tau$ as a metric $\mathfrak{h}$ for the pull-back of $L$ on $X \times \mathbb{P}_{r-1} \times \mathbb{L}$, and we claim that sum of the curvature of $\mathfrak{h}$ and the Ricci curvature is non-negative. To start, $\e^{-\varphi - \alpha(r-1)\chi}$ is non-negatively curved: $\e^{-\varphi - \alpha(r-1)(-\psi+\varphi)}$ contributes semipositively by the hypothesis on curvature of Theorem \ref{thm:SkodaBL}, and by Lagrange's identity we have
\[
  \log\left( |v|^2|h(x)|^2 - |v \cdot h(x)|^2 \right) = \log \sum_{1 \leq i < j \leq r} |h_i(x) v_j - h_j(x) v_i|^2,
\]
which is (locally) plurisubharmonic, being the logarithm of a sum of norms squared of (locally) holomorphic functions (likewise for the right-hand side of the maximum). This leaves us to check that the negativity coming from the factor $(|v|^2)^{\alpha(r-1) - 1}$ is compensated by the Ricci curvature coming from the Fubini--Study volume form: in local coordinates we can write $\frac{1}{|v|^2} |v|^{2\alpha(r-1)} \dVFS$ as $\frac{\dif V(z)}{(1 + |z|^2) ^{r+1 - \alpha(r-1)}}$ (where $\dif V(z)$ is the standard Euclidean volume form), which is positively curved for $\alpha < \frac{r+1}{r-1}$.

\subsection{Extrema of the family of norms}\label{subsec:extrema}
We now investigate the behavior of the family of norms $\norm{\cdot}_t$ near $t=0$ and as $t \rightarrow -\infty$.

At the $t=0$ extreme we have $\chi_0(x,v) = \log\left(|v|^2 |h(x)|^2 \e^{-\psi+\varphi}\right)$, so that
\begin{equation}\label{eq:norm0}
  \begin{split}
    \norm{\hat{s}}_0^2 &= \frac{r!}{\pi^{r-1}} \int_{X \times \mathbb{P}_{r-1}} \frac{|v \cdot s|^2 \e^{-\varphi}}{|v|^2} \wedge \frac{\dVFS}{(|h|^2 \e^{-\psi+\varphi})^{\alpha(r-1)}} \\
    &= \int_X \frac{|s|^2 \e^{-\varphi}}{(|h|^2 \e^{-\psi+\varphi})^{\alpha(r-1)}} = \norm{s}_{E^{\oplus r}}^2,
  \end{split}
\end{equation}
which recovers the norm-squared of $s$ before the lifting. Consequently, for $t=0$, lifting functionals also preserves norms:
\[
  \norm{\hat{\xi}_\eta}_{0,*} = \sup_{\hat{s} \in H^0(X \times \mathbb{P}_{r-1}, L)} \frac{|\hat{\xi}_\eta(\hat{s})|}{\norm{\hat{s}}_0} = \sup_{s \in H^0(X,E^{\oplus r} \otimes K_X)} \frac{|\xi_\eta(s)|}{\norm{s}_{E^{\oplus r}}} = \norm{\xi_\eta}_*.
\]

We now turn to the other extreme of the family, i.e.~$t \rightarrow -\infty$. Fix $x \in X$ and let $A_{t,x}$ be the set of $v \in \mathbb{P}_{r-1}$ such that the maximum in $\chi_t$ is achieved by the second entry, i.e.
\[
  A_{t,x} = \left\{ v \in \mathbb{P}_{r-1} \, \middle| \, 1 - \frac{|v \cdot h(x)|^2}{|v|^2 |h(x)|^2} < \e^t \right\}.
\]
By choosing homogeneous coordinates so that $v_1$ is parallel to $h(x)$ (which is not 0 since we are assuming that the $h_i$'s have no common zeros), and by choosing local coordinates so that $v_1=1$, one sees that $A_{t,x}$ is a ball of real dimension $2r-2$, centered at $[h(x)]$ (the origin, in local coordinates) and of radius $\sqrt{\frac{\e^t}{1-\e^t}} \underset{t \rightarrow -\infty}{\sim} \e^{t/2}$.

We can then split $\norm{\hat{s}}_t^2$ into two summands:
\[
  \norm{\hat{s}}_t^2 = \int_X \one_{t,s}(x) + \int_X \two_{t,s}(x),
\]
where
\[
  \one_{t,s}(x) := \frac{r!}{\pi^{r-1}} \frac{\e^{-(r-1)t}}{(|h|^2 \e^{-\psi+\varphi})^{\alpha(r-1)}} \int_{A_{t,x}} \frac{|v \cdot s(x)|^2 \e^{-\varphi}}{|v|^2} \wedge \dVFS
\]
and
\[
  \two_{t,s}(x) := \frac{r!}{\pi^{r-1}} \frac{\e^{-(r-1)t}}{(|h|^2 \e^{-\psi+\varphi})^{\alpha(r-1)}} \int\displaylimits_{\mathbb{P}_{r-1} \setminus A_{t,x}} \frac{|v \cdot s(x)|^2 \e^{-\varphi} \e^{\alpha(r-1)t}}{|v|^2 \left( 1 - \frac{|v \cdot h|^2}{|v|^2|h|^2} \right)^{\alpha(r-1)}} \wedge \dVFS.
\]

For the first term we get
\begin{equation}\label{eq:It}
  \lim_{t \rightarrow -\infty} \one_{t,s}(x) = r \frac{|h \dototimes s|^2 \e^{-\psi}}{(|h|^2 \e^{-\psi+\varphi})^{\alpha(r-1) + 1}},
\end{equation}
since asymptotically
\[
  \int_{A_{t,x}} \frac{|v \cdot s(x)|^2 \e^{-\varphi}}{|v|^2} \wedge \dVFS \underset{t \rightarrow -\infty}{\sim} \frac{\pi^{r-1}}{(r-1)!} \e^{(r-1)t} \frac{|h \dototimes s|^2 \e^{-\varphi}}{|h|^2}.
\]

For the second
\[
  \two_{t,s}(x) = \e^{-(r-1)t} \int_t^0 \e^{-\alpha(r-1)(\tilde{t} - t)} \dif \nu_{x,s}(\tilde{t}),
\]
with
\[
  \nu_{x,s}(t) := \frac{r!}{\pi^{r-1}} \frac{1}{(|h|^2 \e^{-\psi+\varphi})^{\alpha(r-1)}} \int_{A_{t,x}} \frac{|v \cdot s|^2 \e^{-\varphi}}{|v|^2} \wedge \dVFS = \e^{(r-1)t} \one_{t,s}(x).
\]
Clearly $\nu_{x,s}$ is increasing and positive. Moreover, by \eqref{eq:It},
\[
  \lim_{t \rightarrow -\infty} \e^{-(r-1)t} \nu_{x,s}(t) = r \frac{|h \dototimes s|^2 \e^{-\psi}}{(|h|^2 \e^{-\psi+\varphi})^{\alpha(r-1) + 1}}.
\]
Hence, by Lemma \ref{lemma:calculusEstimate},
\begin{equation}\label{eq:IIt}
  \lim_{t \rightarrow -\infty} \two_{t,s}(x) = r \frac{|h \dototimes s|^2 \e^{-\psi}}{(|h|^2 \e^{-\psi+\varphi})^{\alpha(r-1) + 1}} \frac{r-1}{\alpha(r-1) - (r-1)}.
\end{equation}

All in all
\begin{equation}\label{eq:infinity-extremum}
  \lim_{t \rightarrow -\infty} \norm{\hat{s}}_t^2 = r \frac{\alpha}{\alpha-1} \int_X \frac{|h \dototimes s|^2 \e^{-\psi}}{(|h|^2 \e^{-\psi+\varphi})^{\alpha(r-1) + 1}} = r \frac{\alpha}{\alpha - 1} \norm{h \dototimes s}_G^2,
\end{equation}
retrieving (a multiple of) the norm-squared of the image $h \dototimes s$ of $s$.

\subsection{Monotonicity of the family of dual norms and end of the proof}\label{subsec:convexity}
Now that we have a metric $\mathfrak{h}$ for (the pull-back of) $L$ on $X \times \mathbb{P}_{r-1} \times \mathbb{L}$ with positive enough curvature, Berndtsson's Theorem \ref{thm:Berndtsson} gives the core step of the argument. Fix $\eta \in C^\infty_c(X,G \otimes K_X)$.

\begin{lemma}\label{lemma:convexity}
  The function
  \[
    \begin{matrix}
      (-\infty,0] & \longrightarrow & \mathbb{R}                         \\
      t           & \longmapsto     & \log \norm{\hat{\xi}_\eta}_{t,*}^2
    \end{matrix}
  \]
  is non-decreasing. In particular,
  \[
    \norm{\xi_\eta}_*^2 = \norm{\hat{\xi}_\eta}_{0,*}^2 \geq \norm{\hat{\xi}_\eta}_{t,*}^2 \quad \text{for all } t \leq 0.
  \]
\end{lemma}

\begin{proof}
  \textbf{Step 1.} We first prove that $\sup_{\tau \in \mathbb{L}} \norm{\hat{\xi}_\eta}_{\tau,*}^2 < +\infty$. Once more, $\norm{\hat{\xi}_\eta}^2_{\tau,*}$ only depends on $\Re \tau =: t$, so it suffices to prove that $\norm{\hat{\xi}_\eta}^2_{t,*}$ is uniformly bounded for all $t$ sufficiently negative (the other $t$'s are in a compact interval, so the corresponding norms are automatically uniformly bounded).

  Let
  \[
    C_\eta := \int_X \frac{|\eta|^2 \e^{-\psi}}{(|h|^2 \e^{-\psi+\varphi})^{\alpha(r-1)+1}} < +\infty,
  \]
  then
  \[
    \norm{\hat{\xi}_\eta}_{t,*}^2 = \sup_{\norm{\hat{s}}_t^2 = 1} \left| \int_X \frac{(h \dototimes s) \bar{\eta} \e^{-\psi}}{(|h|^2 \e^{-\psi+\varphi})^{\alpha(r-1)+1}} \right|^2 \leq C_\eta \sup_{\norm{\hat{s}}_t^2 = 1} \int_X \frac{|h \dototimes s|^2 \e^{-\psi}}{(|h|^2 \e^{-\psi+\varphi})^{\alpha(r-1)+1}}.
  \]

  By \eqref{eq:It}, if $t$ is sufficiently negative,
  \[
    \frac{|h \dototimes s|^2 \e^{-\psi}}{(|h|^2 \e^{-\psi+\varphi})^{\alpha(r-1)+1}} \leq \frac{2}{r} \one_{t,s}(x)
  \]
  (in the sense of top forms), so that
  \[
    \norm{\hat{\xi}_\eta}_{t,*}^2 \leq \frac{2 C_\eta}{r} \sup_{\norm{\hat{s}}_t^2=1} \int_X \one_{t,s}(x) \leq \frac{2 C_\eta}{r} < +\infty,
  \]
  as wanted.

  \textbf{Step 2.} Consider now the trivial fibration $(X \times \mathbb{P}_{r-1}) \times \mathbb{L} \overset{\pr_\mathbb{L}}{\longrightarrow} \mathbb{L}$. We have already checked at the end of Subsection \ref{subsec:setup} that the curvature of $\mathfrak{h}$, seen as a metric for $\pr_{X \times \mathbb{P}_{r-1}}^* L \rightarrow (X \times \mathbb{P}_{r-1}) \times \mathbb{L}$, plus the Ricci curvature coming from the Fubini--Study volume form is non-negative on the total space $X \times \mathbb{P}_{r-1} \times \mathbb{L}$. Then, up to a smooth approximation of $\chi$, Berndtsson's Theorem \ref{thm:Berndtsson} implies that $\tau \mapsto \log\norm{\hat{\xi}_\eta}_{\tau,*}^2$ is subharmonic in $\mathbb{L}$. Since $\norm{\hat{\xi}_\eta}_{\tau,*}$ only depends on $t = \Re \tau$, it follows that $t \mapsto \log\norm{\hat{\xi}_\eta}_{t,*}^2$ is convex on $(-\infty,0)$. If this map decreases anywhere on $(-\infty,0)$, then by convexity we would have $\lim_{t \rightarrow -\infty} \log\norm{\hat{\xi}_\eta}_{t,*}^2 = +\infty$, contradicting the uniform boundedness of $\norm{\hat{\xi}_\eta}_{t,*}^2$ obtained in Step 1. Hence the statement follows.
\end{proof}

Let now $s \in H^0(X, (E \otimes K_X)^{\oplus r})$ be any solution of $h \dototimes s = \mathcal{P} \eta$ (such $s$ exists with bounded $L^2$-norm by the same argument of Proposition \ref{prop:solExistence}). Then by \eqref{eq:infinity-extremum} and Lemma \ref{lemma:convexity} we have
\[
  \begin{split}
    \norm{\xi_\eta}_*^2 &= \norm{\hat{\xi}_\eta}_{0,*}^2 \geq \lim_{t \to -\infty} \norm{\hat{\xi}_\eta}_{t,*}^2 \\
    &\geq \lim_{t \to -\infty} \frac{1}{\norm{\hat{s}}_t^2} \left| \int_X \frac{(h \dototimes s) \overline{\mathcal{P}\eta}\e^{-\psi}}{(|h|^2 \e^{-\psi+\varphi})^{\alpha(r-1)+1}} \right|^2 \\
    &= \lim_{t \to -\infty} \frac{\norm{\mathcal{P}\eta}_G^4}{\norm{\hat{s}}_t^2} = \frac{\alpha-1}{\alpha r} \norm{\mathcal{P}\eta}_G^2
  \end{split}
\]
for all $\eta \in C^\infty_c(X,G \otimes K_X)$. Hence, by \eqref{eq:dualFormulation}, we conclude that the minimal-norm solution $\tilde{f}$ to the division problem $h \dototimes f = g$ has norm
\[
  \norm{\tilde{f}}_{E^{\oplus r}}^2 \leq \norm{g}_G^2 \sup_{\eta \in C^0_c(X,G \otimes K_X)} \frac{\norm{\mathcal{P} \eta}^2_G}{\norm{\xi_\eta}_*^2} \leq r \frac{\alpha}{\alpha-1} \norm{g}_G^2,
\]
proving Theorem \ref{thm:SkodaBL}.

\section{Berndtsson and Lempert's proof of the \texorpdfstring{$L^2$}{L2} extension theorem, revisited}\label{sec:L2-extension}
The proof of Theorem \ref{thm:L2-extension} is in many respects quite similar to the proof of Theorem \ref{thm:SkodaBL} and indeed, as already mentioned in the introduction, the Berndtsson and Lempert approach to the $L^2$ extension theorem served as an inspiration for the proof of Theorem \ref{thm:SkodaBL}. Hence, we will not delve too much into the analogous technical details (that are in any case already described in \cite{BerndtssonLempert2016}) and instead focus on the differences and the general philosophy.

\subsection{Preliminary reductions}
As in the proof of Theorem \ref{thm:SkodaBL}, we can assume that $X$ is a relatively compact domain in some larger Stein manifold and that the metrics involved are smooth (see Section \ref{sec:reductions} and \cite{BerndtssonLempert2016}). Moreover, perhaps after shrinking $X$ further, we can assume that $Z$ meets the boundary of $X$ transversely and that the section $f \in H^0(Z, L \otimes K_Z)$ to be extended is holomorphic up to the boundary of $Z$.

Since the singularities of $Z$ are of codimension at least one in $Z$, they are contained in a hypersurface $H$ of $X$ not containing $Z$. Similarly to what we did to remove the base locus in Section \ref{sec:reductions}, we can then reduce to the case of smooth $Z$ (and $\dif T|_Z \not\equiv 0$) by solving the problem for $Z \setminus H \subset X \setminus H$ and then extending the solution to $X$ by Riemann's Removable Singularities Theorem and the identity principle.

\subsection{Dual formulation of the extension problem}
As in Proposition \ref{prop:solExistence}, we can assume that there is some solution $F \in H^0(X, L \otimes L_Z \otimes K_X)$ to $F|_Z = f \wedge \dif T$ with finite $L^2$-norm. Then, as in Lemma \ref{lemma:minimalNorm}, the norm of the solution $\tilde{F}$ with minimal $L^2$-norm is
\[
  \norm{\tilde{F}}_X^2 = \sup_{g \in C^\infty_c(Z, L|_Z \otimes K_Z)} \frac{|\xi_g(f) |^2}{\norm{\xi_g}_*^2},
\]
where $\xi_g$ is the functional that associates
\[
  \xi_g(s) := \int_Z \sigma \bar{g} \e^{-\varphi}
\]
to $s \in H^0(X, L \otimes L_Z \otimes K_X)$, with $s|_Z =: \sigma \wedge \dif T$. Therefore, 
\[
  \norm{\tilde{F}}_X^2 \leq \norm{f}_Z^2 \sup_{g \in C^\infty_c(Z, L|_Z \otimes K_Z)} \frac{\norm{\mathcal{P}g}_Z^2}{\norm{\xi_g}_*^2},
\]
where 
\[
  \mathcal{P}: L^2(Z, L|_Z \otimes K_Z) \longrightarrow H^0(Z, L|_Z \otimes K_Z) \cap L^2(Z, L|_Z \otimes K_Z)
\]
denotes the Bergman projection. Thus it suffices to prove that 
\begin{equation}\label{eq:L2-dualFormulation}
  \norm{\mathcal{P}g}_Z^2 \leq \pi \left( 1 + \frac{1}{\delta} \right) \norm{\xi_g}_*^2
\end{equation}
for all $g \in C_c^\infty(Z, L|_Z \otimes K_Z)$ (see also (3.4) in \cite{BerndtssonLempert2016}).

\subsection{The family of metrics}
We now define a family of metrics for $L \otimes L_Z \rightarrow X$ by introducing a weight $\chi_\tau$ that ``collapses'' $X$ onto $Z$:
\[
  \chi_\tau := \max( \log|T|^2 - \lambda - \Re \tau, 0).
\]
Of course this function only depends on $\Re\tau =: t$ and thus in the following we will use the notation $\chi_t$ instead. We then obtain a corresponding family of metrics $\mathfrak{h}_\tau$ for $L \otimes L_Z \rightarrow X$ by setting
\[
  \mathfrak{h}_\tau := \e^{-\Re \tau} \e^{-\varphi-\lambda} \e^{-(1+\delta)\chi_\tau},
\]
and a family of norms for sections $s \in H^0(X, L \otimes L_Z \otimes K_X)$ given by
\[
  \norm{s}_t^2 := \e^{-t} \int_X |s|^2 \e^{-\varphi-\lambda} \e^{-(1+\delta)\chi_t}.
\]
We will interpret $\mathfrak{h}_\tau$ as a metric $\mathfrak{h}$ for $\pr_X^* (L \otimes L_Z) \rightarrow X \times \mathbb{L}$ (recall that $\mathbb{L}$ denotes the left complex half-plane). Notice that $\mathfrak{h}$ has non-negative curvature by the hypotheses of Theorem \ref{thm:L2-extension}.

\begin{remark}
  Compared to the metrics used in \cite{BerndtssonLempert2016}, the multiplicative constant in front of $\chi_t$ is the \emph{fixed} value $1+\delta$ and thus we will not be able to send it to infinity. Rather than being a problem, this and Lemma \ref{lemma:calculusEstimate} are what keep the curvature of $\mathfrak{h}$ under control without assuming that $\I\ddbar\lambda\geq0$ (see also \cite{Nguyen2023,NguyenWang2023} for similar observations).
\end{remark}

We will denote by $\norm{\cdot}_{t,*}$ the induced dual norms on linear functionals of 
\[
  L^2(X, L \otimes L_Z \otimes K_X) \cap H^0(X, L \otimes L_Z \otimes K_X).
\]

\subsection{Extrema of the family of norms}
Clearly one has $\norm{s}_0^2 = \norm{s}_X^2$. To study the other extremum $t \rightarrow -\infty$, fix $t < 0$ and consider the set $A_t$ of points in $X$ for which the maximum in $\chi_t$ is attained by 0:
\[
  A_t := \left\{ x \in X \, \middle| \, |T(x)|^2 \e^{-\lambda} \leq \e^t \right\}.
\]
We then write  
\[
  \norm{s}_t^2 = \e^{-t} \int_{A_t} |s|^2 \e^{-\varphi-\lambda} + \e^{-t} \int_{X \setminus A_t} |s|^2 \e^{-\varphi-\lambda} \left(\frac{\e^t}{|T|^2 \e^{-\lambda}}\right)^{1+\delta}
\]
and proceed to estimate the two summands as $t \rightarrow -\infty$.

Notice first that the set $A_t$ collapses to $Z$ as $t \rightarrow -\infty$. More precisely, it asymptotically resembles a tube about $Z$ of radius-squared around each $z \in Z$ asymptotic to $\frac{\e^t}{|\dif T(z)|^2 \e^{-\lambda}}$. Write $s|_Z =: \sigma \wedge \dif T$, then 
\begin{equation}\label{eq:L2-It}
  \lim_{t \rightarrow -\infty} \e^{-t} \int_{A_t} |s|^2 \e^{-\varphi-\lambda} = \pi \int_Z |\sigma|^2 \e^{-\varphi}.
\end{equation}

As in Subsection \ref{subsec:extrema}, the second integral can be rewritten as 
\[
  \e^{-t} \int_t^0 \e^{-(1+\delta)(\tilde{t} - t)} \dif \nu_s(\tilde{t}), \quad \text{with } \nu_s(t) := \int_{A_t} |s|^2 \e^{-\varphi-\lambda}.
\]
Since $\nu_s$ is clearly positive increasing and satisfies $\lim_{t \rightarrow -\infty} \e^{-t} \nu_s(t) = \pi \int_{Z} |\sigma|^2 \e^{-\varphi}$, Lemma \ref{lemma:calculusEstimate} gives 
\[
  \lim_{t \rightarrow -\infty} \e^{-t} \int_t^0 \e^{-(1+\delta)(\tilde{t} - t)} \dif \nu_s(\tilde{t}) = \frac{\pi}{\delta} \int_Z |\sigma|^2 \e^{-\varphi}.
\]

All in all one gets 
\begin{equation}\label{eq:L2-liminf}
  \lim_{t \rightarrow -\infty} \norm{s}_t^2 = \pi \left( 1 + \frac{1}{\delta} \right) \int_Z |\sigma|^2 \e^{-\varphi}.
\end{equation}

\subsection{Monotonicity of the family of dual norms}
As in Lemma \ref{lemma:convexity} (or \cite[Lemma 3.2]{BerndtssonLempert2016}), to prove that 
\[
  \norm{\xi_g}_*^2 = \norm{\xi_g}_{0,*}^2 \geq \norm{\xi_g}_{t,*}^2 \quad \text{for all } t \leq 0
\]
it suffices to show that $\sup_{t \leq 0} \norm{\xi_g}_{t,*}^2 < +\infty$. Berndtsson's Theorem \ref{thm:Berndtsson} will then give the required estimate. Let  
\[
  C_g := \int_Z |g|^2 \e^{-\varphi} < +\infty
\]
(recall that $g$ has compact support in $Z$), then 
\[
  \norm{\xi_g}_{t,*}^2 = \sup_{\norm{s}_t = 1} \left| \int_Z \sigma \bar{g} \e^{-\varphi} \right|^2 \leq C_g \sup_{\norm{s}_t = 1} \int_Z |\sigma|^2 \e^{-\varphi}.
\]

As in Lemma \ref{lemma:convexity}, it suffices to check uniform boundedness for $t$ less than some very negative fixed $t_0$. If this is the case, by \eqref{eq:L2-It} one gets
\[
  \int_Z |\sigma|^2 \e^{-\varphi} \leq \frac{2}{\pi} \e^{-t} \int_{A_t} |s|^2 \e^{-\varphi -\lambda} \leq \frac{2}{\pi} \e^{-t} \int_X |s|^2 \e^{-\varphi-\lambda} \e^{-(1+\delta)\chi_t},
\]
so that $\norm{\xi_g}_{t,*}$ is bounded by some uniform constant, as wanted.

\subsection{End of the proof}
To conclude the proof it is enough to show that
\begin{equation}\label{eq:L2-finalEstimate}
  \lim_{t \rightarrow -\infty} \norm{\xi_g}_{t,*}^2 \geq \frac{\delta}{\pi(1+\delta)} \norm{\mathcal{P}g}_Z^2
\end{equation}
for all $g \in C^\infty_c(Z, L|_Z \otimes K_Z)$. Indeed, if \eqref{eq:L2-finalEstimate} holds one has 
\[
  \norm{\mathcal{P}g}_Z^2 \leq \pi \left( 1 + \frac{1}{\delta} \right) \lim_{t \rightarrow -\infty} \norm{\xi_g}_{t,*}^2 \leq \pi \left( 1 + \frac{1}{\delta} \right) \norm{\xi_g}_*^2,
\]
which is what is needed to prove Theorem \ref{thm:L2-extension} by \eqref{eq:L2-dualFormulation}.

Let $s \in H^0(X, L \otimes L_Z \otimes K_X)$ be any finite-norm extension of $\mathcal{P}g$ (so that $s|_Z = \mathcal{P}g \wedge \dif T$). Then by \eqref{eq:L2-liminf}
\[
  \lim_{t \to -\infty} \norm{\xi_g}_{t,*}^2 \geq \lim_{t \to -\infty} \frac{1}{\norm{s}_t^2} \left|\int_Z |\mathcal{P}g|^2 \e^{-\varphi}\right|^2 = \lim_{t \to -\infty} \frac{\norm{\mathcal{P}g}_Z^4}{\norm{s}_t^2} \geq \frac{\delta}{\pi(1+\delta)} \norm{\mathcal{P}g}_Z^2.
\]
This concludes the proof of Theorem \ref{thm:L2-extension}.

\section{Remarks on extension in higher codimension}\label{sec:L2-higherCodim}
When the subvariety $Z$ has codimension $k$ higher than 1, the adjoint formulation does not fit the extension problem as well as in the hypersurface case of Theorem \ref{thm:L2-extension}.

A special context in which formulating extension in terms of canonical sections makes sense is the setting of the Ohsawa--Takegoshi--Manivel Theorem \cite{Manivel1993,Demailly2000}. Assume that $Z$ is cut out by a holomorphic section $T$ of some holomorphic vector bundle $E \rightarrow X$ of rank $k$ and that $T$ is generically transverse to the zero section of $E$. This means that we know a priori that the normal bundle of $Z$ in $X$ extends to the vector bundle $E \rightarrow X$ and that we have the adjunction formula $(K_X \otimes \det E)|_Z = K_Z$. Assume also that $\sup_X h(T,\bar{T}) \leq 1$ for some metric $h$ for $E \rightarrow X$.

Then everything goes through in the same way as Theorem \ref{thm:L2-extension}, up to replacing $L_Z$ with $\det E$ and adapting the curvature assumptions. Explicitly, assume that 
\[
  \I\ddbar\varphi \geq \I\ddbar\log\det h
\]
and
\[
  \I\ddbar\varphi \geq \I\ddbar\log\det h - (k+\delta) \I\ddbar\log h(T,\bar{T}).
\]
Then for every holomorphic section $f \in H^0(Z, L|_Z \otimes K_Z)$ such that 
\[
  \int_Z |f|^2 \e^{-\varphi} < +\infty
\]
there is a holomorphic section $F \in H^0(X, L \otimes \det E \otimes K_X)$ such that $F|_Z = f \wedge \det (\dif T)$ and 
\[
  \int_X |F|^2 \e^{-\varphi + \log\det h} \leq \sigma_k \left( 1 + \frac{k}{\delta} \right) \int_Z |f|^2 \e^{-\varphi},
\]
where $\sigma_k := \pi^k/k!$ is the volume of the unit ball in real dimension $2k$. Here the weight is the function $\chi_\tau := \max\left(\log h(T,\bar{T}) - \Re\tau, 0\right)$ and the family of metrics is 
\[
  \e^{-k\Re\tau} \e^{-\varphi+\log\det h} \e^{-(k+\delta)\chi_\tau}.
\]
Notice that this recovers Theorem \ref{thm:L2-extension} with $E = L_Z$ and $h = \e^{-\lambda}$.

Unfortunately in general we cannot assume that $Z$ is cut out by a section of some vector bundle. In such case no clear analogue of adjunction is available and thus the non-adjoint formulation is preferable.

\begin{maintheorem}[\textbf{non-adjoint $L^2$ extension}]\label{thm:L2-nonAdjoint}
  Let $X$ be a Stein manifold with K\"ahler form $\omega$ and let $\rho: X \rightarrow [-\infty,0]$ be such that
  \[
    \log \dist_Z^2 - \beta \leq \rho \leq \log \dist_Z^2 + \alpha
  \]
  for some smooth function $\beta$ on $X$ and some constant $\alpha$ \paren{so that $Z = \{\rho = -\infty\}$}. Fix $\delta>0$, let $L \rightarrow X$ be a holomorphic line bundle with metric $\e^{-\varphi}$, and assume that 
  \[
    \I\ddbar\varphi + \Ric_\omega \geq 0, \quad \I\ddbar\varphi + \Ric_\omega + (k+\delta) \I\ddbar \rho \geq 0.
  \]
  Then for every holomorphic section $f \in H^0(Z,L|_Z)$ such that 
  \[
    \norm{f}_Z^2 := \int_Z |f|^2 \e^{-\varphi + k\beta} \dif V_Z < +\infty
  \]
  there is a holomorphic section $F \in H^0(X,L)$ such that $F|_Z = f$ and 
  \[
    \norm{F}_X^2 := \int_Z |F|^2 \e^{-\varphi} \dif V_X \leq \sigma_k \left( 1 + \frac{k}{\delta} \right) \norm{f}_Z^2.
  \]
\end{maintheorem}

By taking $X$ to be a pseudoconvex domain $D \subset \mathbb{C}^n$, $\rho = G - \psi$, $\varphi = \phi - k\psi$ and $\beta = B + \psi$ one obtains Theorem 3.8 in \cite{BerndtssonLempert2016} (without the assumption that $G$ and $\psi$ are plurisubharmonic).

The argument is the same as the one for Theorem \ref{thm:L2-extension}, except that the functionals $\xi_g$ are now
\[
  \xi_g(s) := \int_Z s \bar{g} \e^{-\varphi + k\beta} \dif V_Z
\]
and the weights are $\chi_\tau := \max (\rho - \Re\tau,0)$. Then the family of norms becomes
\[
  \norm{s}_t^2 := \e^{-kt} \int_X |s|^2 \e^{-\varphi} \e^{-(k+\delta)\chi_t} \dif V_X.
\]

Clearly $\norm{s}_0^2 = \norm{s}_X^2$. For the other end of the family, let $A_t := \{ \varphi < t \}$, then 
\[
  \norm{s}_t^2 = \e^{-kt} \int_{A_t} |s|^2 \e^{-\varphi} \dif V_X + \e^{-kt} \int_{X \setminus A_t} |s|^2 \e^{-\varphi} \e^{-(k+\delta)(\rho-t)} \dif V_X.
\]
Since the set $A_t$ is asymptotic to a tube around $Z$ of radius-squared bounded above by $\e^{t+\beta}$, it follows that  
\[
  \lim_{t \rightarrow -\infty} \e^{-kt} \int_{A_t} |s|^2 \e^{-\varphi} \dif V_X \leq \sigma_k \int_Z |s|^2 \e^{-\varphi+k\beta} \dif V_Z
\]
and, by Lemma \ref{lemma:calculusEstimate},
\[
  \lim_{t \rightarrow -\infty} \e^{-kt} \int_{X \setminus A_t} |s|^2 \e^{-\varphi} \e^{-(k+\delta)(\rho-t)} \dif V_X \leq \frac{k\sigma_k}{\delta} \int_Z |s|^2 \e^{-\varphi+k\beta} \dif V_Z.
\]
All in all 
\[
  \lim_{t \rightarrow -\infty} \e^{-kt} \int_X |s|^2 \e^{-\varphi} \e^{-(k+\delta)\chi_t} \dif V_X \leq \sigma_k \left( 1 + \frac{k}{\delta} \right) \int_Z |s|^2 \e^{-\varphi+k\beta} \dif V_Z,
\]
which is what is needed to prove Theorem \ref{thm:L2-nonAdjoint}.

\printbibliography

\end{document}